\documentclass[a4paper,10pt]{article}
\usepackage[top=2.5cm,bottom=2.5cm,left=2.5cm,right=2.5cm]{geometry}
\usepackage{cite,amsmath,amssymb}
    \usepackage[margin=1cm,%
                font=small,%
                format=hang,%
                labelsep=period,%
                labelfont=bf]{caption}
\usepackage[algoruled,linesnumbered]{algorithm2e}
\usepackage{algorithmic}

\usepackage{amsfonts,epsf,here,graphicx}
\usepackage{latexsym,multirow,rotating}
\usepackage{pgf,tikz,subfigure}
\usepackage{epstopdf}

\newtheorem{theorem}{\bf Theorem}[section]

\newtheorem{lemma}[theorem]{\bf Lemma}
\newtheorem{proposition}[theorem]{\bf Proposition}

\newcommand{\qed}{\hfill $\square$ \bigskip}

\begin{document}

\baselineskip=0.30in
\vspace*{40mm}

\begin{center}
{\LARGE \bf On the Steiner hyper-Wiener index of a graph}
\bigskip \bigskip

{\large \bf Niko Tratnik
}
\bigskip\bigskip

\baselineskip=0.20in

\textit{Faculty of Natural Sciences and Mathematics, University of Maribor, Slovenia} \\
{\tt niko.tratnik1@um.si}
\medskip

\bigskip\medskip

(Received May 15, 2018)

\end{center}

\noindent
\begin{center} {\bf Abstract} \end{center}

\vspace{3mm}\noindent
In this paper, we study the Steiner hyper-Wiener index of a graph, which is obtained from the standard hyper-Wiener index by replacing the classical graph distance with the Steiner distance. It is shown how this index is related to the Steiner Hosoya polynomial, which generalizes similar result for the standard hyper-Wiener index. Next, we show how the Steiner $3$-hyper-Wiener index of a modular graph  can be expressed by using the classical graph distances. As the main result, a method for computing this index for median graphs is developed. Our method makes computation of the Steiner $3$-hyper-Wiener index much more efficient. Finally, the method is used to obtain the closed formulas for the Steiner $3$-Wiener index and the Steiner $3$-hyper-Wiener index of grid graphs. 


\baselineskip=0.30in


 \medskip\noindent
 {\bf AMS Subj. Class:} 92E10; 05C12; 05C31; 05C90

\section{Introduction}

The Wiener index and the hyper-Wiener index are distance-based graph invariants, used as structure descriptors for predicting physico–chemical properties of organic compounds (often those significant for chemistry, pharmacology, agriculture, environment-protection etc.). Their history goes back to $1947$, when H. Wiener used the distances in the molecular graphs of alkanes to calculate their boiling points \cite{Wiener}. This research has led to the {Wiener index}, which is defined as
$$W(G) = \sum_{\lbrace u, v \rbrace \subseteq V(G)} d(u,v)$$

\noindent
for any connected graph $G$. The origins and applications of the Wiener index are discussed in \cite{rouvray}, while some recent research related to this index can be found in \cite{cai,krnc,lan,lei}. 

The hyper-Wiener index was introduced in 1993 by M. Randi\' c \cite{randic} and has been extensively studied in many papers (see, for example, \cite{chen,klavzar-2000,zigert-2000}). Randi\' c's original definition of the hyper-Wiener index was applicable just to trees and therefore, the {hyper-Wiener index} was later defined for any connected graph $G$ \cite{klein} as
$$WW(G) = \frac{1}{2}\sum_{\lbrace u, v \rbrace \subseteq V(G)} d(u,v) + \frac{1}{2}\sum_{\lbrace u, v \rbrace \subseteq V(G)} d(u,v)^2.$$

The hyper-Wiener index is related to the Hosoya polynomial \cite{hosoya}, which is defined as
$$H(G,x) = \sum_{m \geq 0} d(G,m)x^m,$$
where $d(G,m)$ is the number of unordered pairs of vertices at distance $m$. In \cite{cash} the following relation was shown:
$$WW(G) = H'(G,1) + \frac{1}{2}H''(G,1).$$

Distance-based topological indices were extensively investigated and also the edge versions were studied \cite{trat-zi}. Many methods for computing these indices more efficiently were proposed and the most famous between them is the cut method \cite{klavzar-2015}. This method is commonly used on benzenoid systems \cite{cre-trat1,tratnik2} or on partial cubes \cite{cre-trat}, which constitute a large class of graphs with a lot of applications and includes, for example, many families of chemical graphs (benzenoid systems, trees, phenylenes, cyclic phenylenes, polyphenylenes). In particular, methods for computing the hyper-Wiener index and the edge-hyper-Wiener index were introduced in \cite{sklavzar-2000,tratnik1}.

The Steiner distance of a graph, introduced by
Chartrand et.\,al.\,in 1989 \cite{char}, is a natural and nice generalization of the concept of classical graph distance. For a set $S \subseteq V(G)$, the Steiner distance $d(S)$ among the vertices of $S$ is the minimum size among all connected subgraphs whose vertex sets contain $S$. The Steiner distance in a graph was considered in many papers, for some relevant investigations see \cite{beineke-1996,eroh,oe-1999,oe-ti-1990,yeh-chi-2008} and a survey paper \cite{mao}. On the other hand, the Steiner tree problem requires a tree $T$ with minimum number of edges such that $S \subseteq V(T)$. In general, this problem is known to be NP-complete \cite{hwang}. The Steiner distance and the Steiner tree problem have a lot of applications in real-word problems, for example in circuit layout, network design and in modelling of biomolecular structures \cite{mon}.

If in the definition of the Wiener index the classical graph distance  is replaced by the Steiner distance, the Steiner Wiener index is obtained \cite{li}. In particular, if we replace the distances between pairs of vertices by the distances of all subsets with cardinality $k$, we obtain the Steiner $k$-Wiener index. It was shown in \cite{gut1} that for some molecules the combination of the Steiner Wiener index and the Wiener index has even better correlation with the boiling points than the Wiener index. For some recent investigations on the Steiner Wiener index see \cite{li1,mao1,mao3}. However, a closely similar concept was already studied in the past under the name average Steiner distance \cite{dan1,dan2}. Moreover, the Steiner degree distance \cite{gut,mao2} 
and other  analogous generalizations of distance-based molecular descriptors (where the classical distance is replaced by the Steiner distance) were introduced \cite{mao}.

In \cite{li} the relation between the Steiner $3$-Wiener index and the Wiener index was shown for trees and later for modular graphs \cite{kovse} (see the definition in the preliminaries). In this paper, we first prove the relation between the Steiner hyper-Wiener index and the Steiner Hosoya polynomial. Next, we show how the Steiner $3$-hyper-Wiener index of a modular graph  can be expressed by using the classical graph distances. Furthermore, if $G$ is a partial cube, we develop a cut method for computing the Steiner $3$-hyper-Wiener index, which enables us to compute the index very efficiently and also to find the closed formulas for some families of graphs. Finally, our method is used to obtain the closed formulas for the Steiner $3$-Wiener index and the Steiner $3$-hyper-Wiener index of grid graphs.

\section{Preliminaries}

\noindent
Unless stated otherwise, the graphs considered in this paper are simple and finite. For a graph $G$ we say that $|V(G)|$ is the \textit{order} of $G$ and that $|E(G)|$ is its \textit{size}. Moreover, we define $d_G(x,y)$ (or simply $d(x,y)$) to be the usual shortest-path distance between vertices $x, y \in V(G)$.
\bigskip

\noindent
 For a connected graph $G$ and an non-empty set $S \subseteq V(G)$, the \textit{Steiner distance} among the vertices of $S$, denoted by $d_G(S)$ or simply by $d(S)$, is the minimum size among
all connected subgraphs whose vertex sets contain $S$. Note that if $H$ is a connected
subgraph of $G$ such that $S \subseteq V(H)$ and $|E(H)|=d(S)$, then $H$ is a tree. An \textit{$S$-Steiner tree} or a \textit{Steiner tree for $S$}  is a subgraph $T$ of $G$ such that $T$ is a tree and $S \subseteq V(T)$. Moreover, if $T$ is a Steiner tree for $S$ such that $|E(T)|=d(S)$, then $T$ is called a \textit{minimum Steiner tree for $S$}. It is obvious that for a set $S=\lbrace x,y \rbrace$, $x \neq y$, it holds $d(S)=d(x,y)$.
\bigskip

\noindent
Let $G$ be a connected graph and $k$ a positive integer such that $k \leq |V(G)|$. The \textit{Steiner $k$-Wiener index} of $G$, denoted by $SW_k(G)$, is defined as
$$SW_k(G) = \sum_{\substack{S \subseteq V(G) \\ |S|=k}}d(S).$$ 
The \textit{Steiner $k$-hyper-Wiener index} of $G$, denoted by $SWW_k(G)$, is defined as
$$SWW_k(G) = \frac{1}{2}\sum_{\substack{S \subseteq V(G) \\ |S|=k}}d(S) + \frac{1}{2}\sum_{\substack{S \subseteq V(G) \\ |S|=k}}d(S)^2.$$
The \textit{Steiner $k$-Hosoya polynomial} of $G$, denoted by $SH_k(G,x)$, is defined as
$$SH_k(G,x) = \sum_{m\geq 0}d_k(G,m)x^m,$$
where $d_k(G,m)$ denotes the number of subsets $S \subseteq V(G)$ with $|S|=k$ and $d(S)=m$.
\bigskip

\noindent
Two edges $e_1 = u_1 v_1$ and $e_2 = u_2 v_2$ of a connected graph $G$ are in relation $\Theta$, $e_1 \Theta e_2$, if
$$d_G(u_1,u_2) + d_G(v_1,v_2) \neq d_G(u_1,v_2) + d_G(u_1,v_2).$$
Note that this relation is also known as Djokovi\' c-Winkler relation.
The relation $\Theta$ is reflexive and symmetric, but not necessarily transitive.
We denote its transitive closure (i.e.\ the smallest transitive relation containing $\Theta$) by $\Theta^*$.
\bigskip

\noindent
The {\em hypercube} $Q_n$ of dimension $n$ is defined in the following way: 
all vertices of $Q_n$ are presented as $n$-tuples $(x_1,x_2,\ldots,x_n)$ where $x_i \in \{0,1\}$ for each $i, 1\leq i\leq n$, 
and two vertices of $Q_n$ are adjacent if the corresponding $n$-tuples differ in precisely one position. 
\bigskip

\noindent
A subgraph $H$ of a graph $G$ is called an \textit{isometric subgraph} if for each $u,v \in V(H)$ it holds $d_H(u,v) = d_G(u,v)$. Any isometric subgraph of a hypercube is called a {\em partial cube}. For an edge $ab$ of a graph $G$, let $W_{ab}$ be the set of vertices of $G$ that are closer to $a$ than
to $b$. We write $\langle S \rangle$ for the subgraph of $G$ induced by $S \subseteq V(G)$. It is known that for any partial cube $G$ the relation $\Theta$ is transitive, i.e.\,$\Theta=\Theta^*$. Moreover, for any $\Theta$-class  $E$ of a partial cube $G$, the graph $G - E$ has exactly two connected components, namely $\langle W_{ab} \rangle $ and $\langle W_{ba} \rangle$, where $ab \in E$. For more information about $\Theta$ relation and partial cubes see \cite{klavzar-book}. 
\bigskip

\noindent
Let $G$ be a graph. A \textit{median} of a triple of vertices $u,v,w$ of $G$ is a vertex $z$ that lies on a shortest $u,v$-path, on a shortest $u,w$-path and on a shortest $v,w$-path ($z$ can be one of the vertices $u,v,w$). A graph is a \textit{median graph} if every triple of its vertices has a unique median. These graphs were first introduced in \cite{avann} by Avann and arise naturally in the study of ordered sets and distributive lattices. Moreover, a graph is called a \textit{modular graph} if every triple of vertices has at least one median. Obviously, any median graph is also a modular graph. It is known that trees, hypercubes, and grid graphs (Cartesian products of two paths) are median graphs. Moreover, any median graph is a partial cube \cite{klavzar-book}.
\bigskip

\noindent
The following relation between the Steiner 3-Wiener index and the Wiener index was first found out for trees \cite{li} and later generalized to modular graphs \cite{kovse}. 

\begin{theorem} \label{wie} \cite{kovse}
Let $G$ be a modular graph with at least three vertices. Then
$$SW_3(G) =\frac{|V(G)|-2}{2}W(G).$$
\end{theorem}

\section{The Steiner hyper-Wiener index and the Steiner Hosoya polynomial}

In 2002 G. G. Cash showed the relationship between the Hosoya polynomial and the hyper-Wiener index, see \cite{cash}. In the present section we generalize this result and show that it holds also for the Steiner $k$-hyper-Wiener index and the Steiner $k$-Hosoya polynomial where $k$ is a positive integer such that $k \leq |V(G)|$.

First, we notice the obvious connection between the Steiner $k$-Wiener index and the Steiner $k$-Hosoya polynomial.

\begin{proposition} \label{pom1}
Let $G$ be a connected graph and $k$ a positive integer such that $k \leq |V(G)|$. Then
$$SW_k(G) = SH_k'(G,1).$$
\end{proposition}

\noindent
Now we can state the main result of the section.

\begin{theorem}
\label{povezava}
Let $G$ be a connected graph and $k$ a positive integer such that $k \leq |V(G)|$. Then
$$SWW_k(G) = SH_k'(G,1) + \frac{1}{2}SH_k''(G,1).$$
\end{theorem}

\begin{proof}
Since by Proposition \ref{pom1}
$$  \frac{1}{2}\sum_{\substack{S \subseteq V(G) \\ |S|=k}}d(S) = \frac{1}{2}SW_k(G) = \frac{1}{2}SH_k'(G,1),$$
it follows
$$SWW_k(G) = \frac{1}{2}SH_k'(G,1) + \frac{1}{2} \sum_{\substack{S \subseteq V(G) \\ |S|=k}}d(S)^2$$
and it suffices to show
$$\frac{1}{2} \sum_{\substack{S \subseteq V(G) \\ |S|=k}}d(S)^2 = \frac{1}{2}SH_k'(G,1) + \frac{1}{2}SH_k''(G,1).$$
First we notice that
$$\sum_{\substack{S \subseteq V(G) \\ |S|=k}}d(S)^2 = \sum_{m \geq 1}m^2d_k(G,m).$$
Furthermore,
$$(xSH_k'(G,x))' = \left(\sum_{m \geq 1}m \cdot d_k(G,m)x^{m}\right)' = \sum_{m \geq 1}m^2d_k(G,m)x^{m-1}$$
and therefore, if we denote $f(x) = xSH_k'(G,x)$,
$$\sum_{\substack{S \subseteq V(G) \\ |S|=k}}d(S)^2 = f'(1).$$
However,
$$f'(x) = (xSH_k'(G,x))' = SH_k'(G,x) + xSH_k''(G,x)$$
and finally it follows that
$$ \sum_{\substack{S \subseteq V(G) \\ |S|=k}}d(S)^2 = SH_k'(G,1) + SH_k''(G,1),$$
which completes the proof. \qed

\end{proof}
\smallskip

In the rest of this section the previous result is demonstrated on two basic examples. First, let $K_n$ be a complete graph on $n$ vertices and let $k$ be a positive integer such that $k \leq n$. If $S \subseteq V(K_n)$ with $|S|=k$, then $d(S)=k-1$. Therefore
$$SH_k(K_n,x) = {{n}\choose{k}} x^{k-1}.$$
Using Proposition \ref{pom1} and Theorem \ref{povezava} we obtain
$$SW_k(K_n) = (k-1){{n}\choose{k}},$$
$$SWW_k(K_n)= \frac{k(k-1)}{2}{{n}\choose{k}} = {k \choose 2}{{n}\choose{k}}.$$

\noindent
Next, let $P_n$ be a path on $n$ vertices and let $k \geq 2$ be a positive integer such that $k \leq n$. Obviously, for any $j \in \lbrace k-1, \ldots, n-1\rbrace$ we obtain $d_k(P_n,j)=(n-j) {{j-1}\choose{k-2}}$. Therefore
$$SH_k(P_n,x) = \sum_{j=k-1}^{n-1}(n-j){{j-1}\choose{k-2}} x^{j}.$$
Using Proposition \ref{pom1} and Theorem \ref{povezava} we obtain
$$SW_k(P_n) = \sum_{j=k-1}^{n-1}j(n-j){{j-1}\choose{k-2}} = (k-1){{n+1}\choose{k+1}},$$
\begin{eqnarray*}
SWW_k(P_n) &= &  \sum_{j=k-1}^{n-1}\frac{j(j+1)(n-j)}{2}{{j-1}\choose{k-2}} \\
& = & \frac{n(n+1)(n+2)(n-k+1)}{2(k+1)(k+2)} {{n-1} \choose {k-2}} \\
& = & {k \choose 2}{{n+2}\choose{k+2}}.
\end{eqnarray*}

\section{The Steiner $3$-hyper-Wiener index of modular graphs}

Finding a minimum Steiner tree for a given set of vertices $S$ is not easy and it is known to be NP-complete in general \cite{hwang}. Hence, in this section we prove a method for computing the Steiner $3$-hyper-Wiener index of modular graphs. Using this method, computing the Steiner $3$-hyper-Wiener index can be done in polynomial time.

For a connected graph $G$, we will denote by $\overline{WW}(G)$ the sum of squares of distances between all the pairs of vertices, i.e.
$$\overline{WW}(G) = \sum_{\lbrace u,v \rbrace \subseteq V(G)}d(u,v)^2.$$
\noindent Therefore, it is easy to observe that if $G$ is connected,
$$WW(G) = \frac{1}{2}W(G) + \frac{1}{2}\overline{WW}(G).$$

\noindent
 Furthermore, in the rest of the paper the following notations will be used for a graph $G$:
\begin{eqnarray*}
O_2(G) & = & \lbrace (u,v) \in V(G)^2 \, | \, u \neq v  \rbrace, \\
O_3(G) & = & \lbrace (u,v,w) \in V(G)^3 \, | \, u \neq v, u \neq w, v \neq w  \rbrace.
\end{eqnarray*}


In order to obtain the final result of this section, we need two lemmas.
\begin{lemma}
\label{vso_kva}
Let $G$ be a connected graph with at least three vertices. Then
$$\sum_{(u,v,w) \in O_3(G)}d(u,v)^2 = 2(|V(G)|-2)\overline{WW}(G).$$
\end{lemma}

\begin{proof}
First we notice that
$$\sum_{(u,v,w) \in O_3(G)}d(u,v)^2 = \sum_{(u,v) \in O_2(G))} \sum_{\substack{w \in V(G) \\ w \neq u, w \neq v }} d(u,v)^2.$$
Therefore, we get

\begin{eqnarray*}
\sum_{(u,v,w) \in O_3(G)}d(u,v)^2 & = & \sum_{(u,v) \in O_2(G)} \left( d(u,v)^2 \sum_{\substack{w \in V(G) \\ w \neq u, w \neq v }} 1 \right) \\
& = &  \sum_{(u,v) \in O_2(G)} (|V(G)|-2)d(u,v)^2 \\
& = & (|V(G)|-2) \sum_{(u,v) \in O_2(G)} d(u,v)^2 \\
& = & 2(|V(G)|-2)\overline{WW}(G)
\end{eqnarray*}
and the proof is complete. \qed

\end{proof}

\begin{lemma} \label{mod_raz}
Let $G$ be a connected graph with at least three vertices. Three distinct vertices $u,v,w$ have at least one median if and only if for the set  $S= \lbrace u,v,w \rbrace$ it holds
$$2d(S)=d(u,v) + d(u,w) + d(v,w).$$
\end{lemma}

\begin{proof}
Assume that $m$ is a median for three distinct vertices $u,v,w$ and let $S=\lbrace u,v,w \rbrace$. Denote $a=d(u,m)$, $b=d(v,m)$, and $c=d(w,m)$. Since $m$ is a median, it obviously holds $d(u,v) = a+b$, $d(u,w) = a+c$, and $d(v,w)=b+c$. Moreover, let $T$ be a connected subgraph of $G$ containing only a shortest path from $u$ to $m$, a shortest path from $v$ to $m$, and a shortest path from $w$ to $m$, see Figure \ref{Stree}.

\begin{figure}[h!] 
\begin{center}
\includegraphics[scale=0.75]{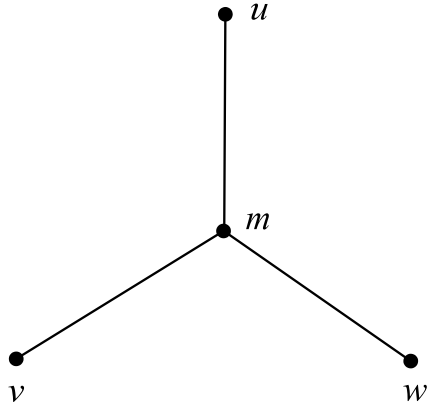}
\end{center}
\caption{\label{Stree} A connected subgraph $T$ which is a minimum Steiner tree.}
\end{figure}

We will first show that $T$ is a minimum Steiner tree for $S$. Let $T'$ be any minimum Steiner tree for $S$ and let $P$ be the path from $u$ to $v$ in $T'$. Also, let $m'$ be the last vertex on the path from $u$ to $w$ in $T'$ that is contained also in $P$. Since $T'$ is a minimum Steiner tree, $T'$ contains only $P$ and the path from $m'$ to $w$ in $T'$. Denote $x=d_{T'}(u,m')$, $y=d_{T'}(v,m')$, and $z=d_{T'}(w,m')$. Obviously, $x+ y \geq a+b$, $x+ z \geq a+c$, and $y+ z \geq b+c$. Therefore, we get $d(S)=|E(T')|= x+y+z \geq a+b+c \geq |E(T)|$ and thus $d(S)=|E(T)|$, which shows that $T$ is a minimum Steiner tree with $|E(T)|=a+b+c$. It follows $d(S) = \frac{1}{2}(d(u,v) + d(u,w) + d(v,w))$.

For the other direction, suppose that $2d(S)=d(u,v) + d(u,w) + d(v,w)$ and let $T'$ be a minimum Steiner tree, i.e.\, $|E(T')|=d(S)$. Let $m'$, $x$, $y$, and $z$ be defined the same as before. It follows $x+y \geq d(u,v)$, $x+z \geq d(u,w)$, and $y+z \geq d(v,w)$. Therefore, $2d(S) = 2(x+y+z) \geq d(u,v)+d(u,w)+d(v,w) = 2d(S)$ and we obtain $x+y = d(u,v)$, $x+z = d(u,w)$, and $y+z = d(v,w)$, which implies that $m'$ is a median for $u,v,w$. \qed
\end{proof}
\smallskip

%
%
%

\noindent
In order to state the main result of the section, the following notation will be used for a connected graph $G$ with at least three vertices:
\begin{equation} \label{oznaka}
\widehat{WW}(G)= \sum_{ (u,v,w) \in O_3(G)} d(u,v)d(u,w).
\end{equation}

\noindent
Finally, we are able to prove how the Steiner 3-hyper-Wiener index of a modular graph $G$ can be obtained from $W(G)$, $\overline{WW}(G)$, and $\widehat{WW}(G)$.

\begin{theorem}
\label{hi-wie}
Let $G$ be a modular graph with at least three vertices. Then
$$SWW_3(G) =\frac{|V(G)|-2}{4}W(G) + \frac{|V(G)|-2}{8}\overline{WW}(G) + \frac{1}{8} \widehat{WW}(G).$$
\end{theorem}

\begin{proof}
Since any three vertices in $G$ have a median, by using Lemma \ref{mod_raz} we calculate

\begin{eqnarray*}
\sum_{\substack{S \subseteq V(G) \\ |S|=3}} d(S)^2 & = & \sum_{\lbrace u,v,w \rbrace \subseteq V(G)} \frac{1}{4}(d(u,v) + d(u,w) + d(v,w))^2 \\
& = & \frac{1}{4} \sum_{\lbrace u,v,w \rbrace \subseteq V(G)} (d(u,v)^2 + d(u,w)^2 + d(v,w)^2) \\
& + & \frac{1}{2} \sum_{\lbrace u,v,w \rbrace \subseteq V(G)} (d(u,v)d(u,w) + d(u,v)d(v,w) + d(u,w)d(v,w)).
\end{eqnarray*}

\noindent
Therefore, since any unordered triple occurs six times as an ordered triple, it follows
\begin{eqnarray*}
\sum_{\substack{S \subseteq V(G) \\ |S|=3}} d(S)^2 &  = & \frac{1}{24} \sum_{(u,v,w) \in O_3(G)} (d(u,v)^2 + d(u,w)^2 + d(v,w)^2) \\
& + & \frac{1}{12} \sum_{(u,v,w) \in O_3(G)} (d(u,v)d(u,w) + d(u,v)d(v,w) + d(u,w)d(v,w)) \\
& = & \frac{1}{8} \sum_{(u,v,w) \in O_3(G)} d(u,v)^2 + \frac{1}{4} \sum_{(u,v,w) \in O_3(G)} d(u,v)d(u,w) \\
& = & \frac{|V(G)|-2}{4}\overline{WW}(G) + \frac{1}{4} \widehat{WW}(G),
\end{eqnarray*}
where the last equality follows by Lemma \ref{vso_kva} and Equation \eqref{oznaka}. Finally, using the definition of the Steiner 3-hyper-Wiener index and Theorem \ref{wie} the proof is complete. \qed
\end{proof}

\section{A cut method for partial cubes}

Theorem \ref{hi-wie} implies that the Steiner 3-hyper-Wiener index of a modular graph can be computed by using $W(G)$, $\overline{WW}(G)$, and $\widehat{WW}(G)$. However, the terms $\overline{WW}(G)$ and $\widehat{WW}(G)$ are not easy to calculate, especially if one is interested in finding closed formulas for a given family of graphs. Therefore, in this section we describe methods for computing $\overline{WW}(G)$ and $\widehat{WW}(G)$ of a partial cube $G$, which are much more efficient than the calculation by the definition.

To prove the results in this section, we need to introduce some additional notation. If $G$ is a partial cube with $\Theta$-classes $E_1, \ldots, E_d$, we denote by $U_i$ and $U_i'$ the connected components of the graph $G - E_i$, where $i \in \lbrace 1, \ldots, d \rbrace$. For any $i,j \in \lbrace 1, \ldots, d \rbrace$, $i \neq j$, set

\begin{eqnarray*}
N_{ij}^{00} & = & V(U_i) \cap V(U_j), \\
N_{ij}^{01} & = & V(U_i) \cap V(U_j'), \\
N_{ij}^{10} & = & V(U_i') \cap V(U_j), \\
N_{ij}^{11} & = & V(U_i') \cap V(U_j').
\end{eqnarray*}

\noindent
Then, for $i,j \in \lbrace 1, \ldots, d \rbrace$, $i \neq j$, and $k,l \in \lbrace 0, 1 \rbrace$ we define
$$n_{ij}^{kl} = |N_{ij}^{kl}|.$$

\noindent
Also, for $i \in \lbrace 1, \ldots, d \rbrace$, let
$$n^0_i = |V(U_i)|, \quad n^1_i = |V(U_i')|.$$

In \cite{sklavzar-2000} it was shown that if $G$ is a partial cube and $d$ is the number of its $\Theta$-classes, then
$$\sum_{u \in V(G)} \sum_{v \in V(G)}d(u,v)^2 = 2W(G) + 4 \sum_{i=1}^{d-1}\sum_{j=i+1}^d \left( n^{00}_{ij}n^{11}_{ij} +  n^{01}_{ij}n^{10}_{ij}\right).$$

\noindent
Moreover, it is known \cite{klavzar-2015} that if $G$ is a partial cube with $d$ $\Theta$-classes, then
\begin{equation} \label{cut_wie}
W(G) = \sum_{i=1}^d n^0_in^1_i.
\end{equation}

\noindent
Taking into account that $\sum_{u \in V(G)} \sum_{v \in V(G)}d(u,v)^2 = 2\sum_{\lbrace u,v \rbrace \subseteq V(G)} d(u,v)^2$ and combining the previous two results, we obtain the following result, which represents a method for computing $\overline{WW}(G)$ for a partial cube $G$.

\begin{proposition} \label{pom}
Let $G$ be a partial cube and let $d$ be the number of its $\Theta$-classes. Then
$$\overline{WW}(G) = \sum_{i=1}^d n^0_in^1_i + 2 \sum_{i=1}^{d-1}\sum_{j=i+1}^d \left( n^{00}_{ij}n^{11}_{ij} +  n^{01}_{ij}n^{10}_{ij}\right).$$
\end{proposition}

On the other hand, a method for computing the term $\widehat{WW}(G)$ is not yet known and therefore, we develop it in the next theorem. If $G$ is a partial cube with $d$ $\Theta$-classes and $U_i,U_i'$, $i \in \lbrace 1, \ldots, d \rbrace$, are defined as before, then for two vertices $u,v \in V(G)$ we set

$$ \delta_i(u,v) = \left\{ \begin{matrix} 1; & u \in V(U_i) \, \& \, v \in V(U_i') \ \text{or} \ u \in V(U_i') \, \& \, v \in V(U_i),\\ 0;& \text{otherwise.}  \end{matrix} \right. $$

\noindent
It is not difficult to check that for any two vertices of $G$ the following equation holds (see, for example, \cite{sklavzar-2000}):
\begin{equation} \label{razd}
d(u,v) = \sum_{i=1}^d \delta_i(u,v).
\end{equation}

\begin{theorem}
\label{vso}
Let $G$ be a partial cube with at least three vertices and let $d$ be the number of its $\Theta$-classes. Then
\begin{eqnarray*}
\widehat{WW}(G)& =& \sum_{i=1}^d \left(n^{0}_{i}n^1_i(n^1_i-1) + n^{1}_{i}n^0_i(n^0_i-1)\right) \\
 & + &2 \sum_{i=1}^{d-1}\sum_{j=i+1}^d \Bigg( 3n^{00}_{ij}n^{01}_{ij}n^{10}_{ij} + 3n^{00}_{ij}n^{01}_{ij}n^{11}_{ij} + 3n^{00}_{ij}n^{10}_{ij}n^{11}_{ij} + 3n^{01}_{ij}n^{10}_{ij}n^{11}_{ij}  \\
 & + &  n^{00}_{ij}n^{11}_{ij}(n^{11}_{ij} - 1) + n^{01}_{ij}n^{10}_{ij}(n^{10}_{ij} - 1) + n^{10}_{ij}n^{01}_{ij}(n^{01}_{ij} - 1) + n^{11}_{ij}n^{00}_{ij}(n^{00}_{ij} - 1) \Bigg).
\end{eqnarray*}
\end{theorem}

\begin{proof}
By Equation \eqref{razd} we get
\begin{eqnarray*}
\sum_{ (u,v,w) \in O_3(G)}\left( d(u,v)d(u,w)\right)& =& \sum_{ (u,v,w) \in O_3(G)}\left[{ \Bigg( \sum_{i=1}^d \delta_i(u,v) \Bigg) \Bigg( \sum_{j=1}^d \delta_j(u,w) \Bigg)} \right] \\
& =& \sum_{ (u,v,w) \in O_3(G)}\left[\sum_{i=1}^d  \sum_{j=1}^d \left( \delta_i(u,v) \delta_j(u,w) \right) \right] \\
& =&\sum_{i=1}^d  \sum_{j=1}^d \left[ \sum_{ (u,v,w) \in O_3(G)}  \delta_i(u,v) \delta_j(u,w)  \right] \\
& =&\sum_{i=1}^d  \left[ \sum_{ (u,v,w) \in O_3(G)}  \delta_i(u,v) \delta_i(u,w)  \right] \\
& +&\sum_{i=1}^d  \sum_{j=1, j \neq i}^d \left[ \sum_{ (u,v,w) \in O_3(G)}  \delta_i(u,v) \delta_j(u,w)  \right].
\end{eqnarray*}

\noindent
Therefore, it follows

\begin{eqnarray*}
\widehat{WW}(G) &   =&\sum_{i=1}^d  \left[ \sum_{ (u,v,w) \in O_3(G)}  \delta_i(u,v) \delta_i(u,w)  \right] \\
& +&2\sum_{i=1}^{d-1}  \sum_{j=i+1}^{d} \left[ \sum_{ (u,v,w) \in O_3(G)}  \delta_i(u,v) \delta_j(u,w)  \right].
\end{eqnarray*}

\noindent
Since $\delta_i(u,v) \delta_i(u,w)=1$ if and only if $\delta_i(u,v)=1$ and $\delta_i(u,w)=1$ (see Figure \ref{moznosti1}), we obtain
$$\sum_{ (u,v,w) \in O_3(G)}  \delta_i(u,v) \delta_i(u,w) = n^{0}_{i}n^1_i(n^1_i-1) + n^{1}_{i}n^0_i(n^0_i-1).$$

\begin{figure}[h!] 
\begin{center}
\includegraphics[scale=0.65]{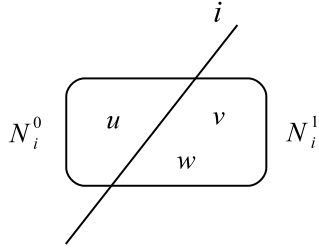}
\end{center}
\caption{\label{moznosti1} Possible position for $u$, $v$, and $w$.}
\end{figure}

Similarly, $\delta_i(u,v) \delta_j(u,w) = 1$ if and only if $\delta_i(u,v)=1$ and $\delta_j(u,w)=1$. First assume that $u,v,w$ belong to two of the sets $N^{00}_{ij},N^{01}_{ij},N^{10}_{ij},N^{11}_{ij}$. If, for example, $u \in N^{00}_{ij}$, the number of triples $(u,v,w)$ satisfying $\delta_i(u,v) \delta_j(u,w) = 1$ is exactly $n^{00}_{ij}n^{11}_{ij}(n^{11}_{ij} - 1)$.

Next, assume that $u,v,w$ belong to three of the sets $N^{00}_{ij},N^{01}_{ij},N^{10}_{ij},N^{11}_{ij}$. If $u \in N^{00}_{ij}$, we have three possibilities for $v$ and $w$, see Figure \ref{moznosti2}. 

\begin{figure}[h!] 
\begin{center}
\includegraphics[scale=0.65]{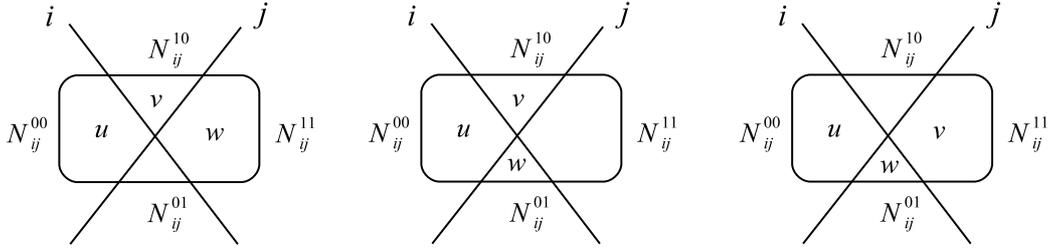}
\end{center}
\caption{\label{moznosti2} Three possibilities for $u$, $v$, and $w$.}
\end{figure}

Hence, the number of triples $(u,v,w)$ satisfying $\delta_i(u,v) \delta_j(u,w) = 1$ and $u \in N^{00}_{ij}$ is exactly $n^{00}_{ij}n^{01}_{ij}n^{10}_{ij} + n^{00}_{ij}n^{01}_{ij}n^{11}_{ij} + n^{00}_{ij}n^{10}_{ij}n^{11}_{ij}$. Similar formulas hold also when $u \in N^{01}_{ij}$, $u \in N^{10}_{ij}$, or $u \in N^{11}_{ij}$. Summing up all the contributions we obtain

\begin{eqnarray*}
 \sum_{ (u,v,w) \in O_3(G)}  \delta_i(u,v) \delta_j(u,w) & = & 3n^{00}_{ij}n^{01}_{ij}n^{10}_{ij} + 3n^{00}_{ij}n^{01}_{ij}n^{11}_{ij} + 3n^{00}_{ij}n^{10}_{ij}n^{11}_{ij} + 3n^{01}_{ij}n^{10}_{ij}n^{11}_{ij}  \\
 & + &  n^{00}_{ij}n^{11}_{ij}(n^{11}_{ij} - 1) + n^{01}_{ij}n^{10}_{ij}(n^{10}_{ij} - 1) + n^{10}_{ij}n^{01}_{ij}(n^{01}_{ij} - 1) + n^{11}_{ij}n^{00}_{ij}(n^{00}_{ij} - 1) 
 \end{eqnarray*}
and the proof is complete. \qed
\end{proof}
\smallskip

\noindent
Finally, we obtain the main result of this paper.

\begin{theorem} \label{glavni}
Let $G$ be a modular graph and a partial cube with at least three vertices and let $d$ be the number of $\Theta$-classes of $G$. Then

\begin{eqnarray*}
SWW_3(G)& =& \frac{3|V(G)|-6}{8}\sum_{i=1}^d n^0_in^1_i + \frac{|V(G)|-2}{4} \sum_{i=1}^{d-1}\sum_{j=i+1}^d \left( n^{00}_{ij}n^{11}_{ij} +  n^{01}_{ij}n^{10}_{ij}\right) \\
& + & \frac{1}{8}\sum_{i=1}^d \left(n^{0}_{i}n^1_i(n^1_i-1) + n^{1}_{i}n^0_i(n^0_i-1)\right) \\
 & + & \frac{1}{4}\sum_{i=i}^{d-1}\sum_{j=i+1}^d \Bigg( 3n^{00}_{ij}n^{01}_{ij}n^{10}_{ij} + 3n^{00}_{ij}n^{01}_{ij}n^{11}_{ij} + 3n^{00}_{ij}n^{10}_{ij}n^{11}_{ij} + 3n^{01}_{ij}n^{10}_{ij}n^{11}_{ij}  \\
 & + &  n^{00}_{ij}n^{11}_{ij}(n^{11}_{ij} - 1) + n^{01}_{ij}n^{10}_{ij}(n^{10}_{ij} - 1) + n^{10}_{ij}n^{01}_{ij}(n^{01}_{ij} - 1) + n^{11}_{ij}n^{00}_{ij}(n^{00}_{ij} - 1) \Bigg).
\end{eqnarray*}

\end{theorem}

\begin{proof}
The result follows by Theorem \ref{hi-wie}, Equation \eqref{cut_wie}, Proposition \ref{pom}, and Theorem \ref{vso}. \qed
\end{proof}
\smallskip

\noindent
Note that the number of $\Theta$-classes of $G$ is much smaller than the number of its edges, therefore, this result enables us to compute the Steiner 3-hyper-Wiener index much more efficiently. In particular, since any median graph is a partial cube and a modular graph, Theorem \ref{glavni} holds for all median graphs. An example showing how this theorem can be used is presented in the next section.

\section{The Steiner 3-hyper-Wiener index of grid graphs}

Let $m$ and $n$ be positive integers. In this paper, the \textit{grid graph} $G_{m,n}$ is defined as the Cartesian product of two paths $P_m$ and $P_n$, i.e.\,$G_{m,n}=P_m \Box P_n$, see Figure \ref{cuts}. Grid graphs are a subfamily of polyomino graphs (a \textit{polyomino graph} consists of a cycle $C$ in the infinite square lattice together with all squares inside $C$), which are known chemical graphs and often arise also in other real-word problems, for an example see \cite{zhou}.

\noindent
Let $G$ be a grid graph (or a polyomino graph) drawn in the plane such that every edge is a line segment of length $1$. An \textit{elementary cut} of $G$ is a line segment that starts at
the center of a peripheral edge of $G$,
goes orthogonal to it and ends at the first next peripheral
edge of $G$. Note that elementary cuts can be defined also for benzenoid systems, where they have been
described and illustrated by numerous examples in several
earlier articles. The elementary cuts of a grid graph $G_{m,n}$ will be denoted by $C_1, \ldots, C_{n-1}$ and $D_1, \ldots,D_{m-1}$, as shown in Figure \ref{cuts}.

\begin{figure}[h!] 
\begin{center}
\includegraphics[scale=0.65]{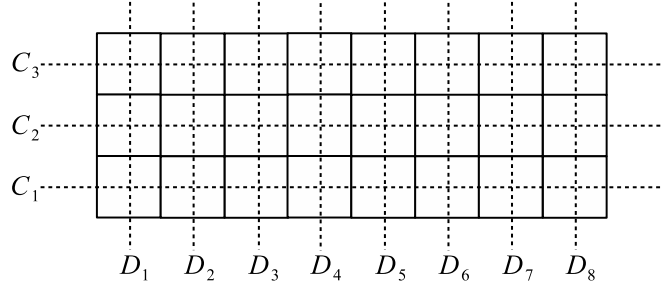}
\end{center}
\caption{\label{cuts} Grid graph $G_{9,4}$ with all the elementary cuts.}
\end{figure}

The main insight for our consideration
is that every $\Theta$-class of a grid graph (or a polyomino graph)
coincides with exactly one of its elementary cuts. Therefore, it is not difficult to check that all grid graphs (or polyomino graphs) are partial cubes.

Moreover, in \cite{kl-sk} it was studied which subgraphs of grid graphs are median graphs. In particular, every grid graph is known to be a median graph. Hence, in this section we find a closed formula for the Steiner $3$-hyper-Wiener index of $G_{m,n}$ by using Theorem \ref{glavni}. 

In order to simplify the computation, we introduce the following notation for a partial cube $G$. For any $\Theta$-class $E_i$ of $G$, let 
\begin{eqnarray*}
f_1(E_i) & = & n^0_in^1_i, \\
f_2(E_i) & = & n^{0}_{i}n^1_i(n^1_i-1) + n^{1}_{i}n^0_i(n^0_i-1)
\end{eqnarray*}
and for any two distinct $\Theta$-classes $E_i$,$E_j$ let
\begin{eqnarray*}
g_1(E_i,E_j) & = & n^{00}_{ij}n^{11}_{ij} +  n^{01}_{ij}n^{10}_{ij}, \\
g_2(E_i,E_j) & = & 3n^{00}_{ij}n^{01}_{ij}n^{10}_{ij} + 3n^{00}_{ij}n^{01}_{ij}n^{11}_{ij} + 3n^{00}_{ij}n^{10}_{ij}n^{11}_{ij} + 3n^{01}_{ij}n^{10}_{ij}n^{11}_{ij}  \\
 & + &  n^{00}_{ij}n^{11}_{ij}(n^{11}_{ij} - 1) + n^{01}_{ij}n^{10}_{ij}(n^{10}_{ij} - 1) + n^{10}_{ij}n^{01}_{ij}(n^{01}_{ij} - 1) + n^{11}_{ij}n^{00}_{ij}(n^{00}_{ij} - 1).
\end{eqnarray*}

\noindent
Moreover, if $E_1, \ldots, E_d$ are all the $\Theta$-classes of $G$, we define

\begin{eqnarray*}
S_1(G) & = & \sum_{i=1}^d f_1(E_i), \\
S_2(G) & = & \sum_{i=1}^d f_2(E_i), \\
S_3(G) & = & \sum_{i=1}^{d-1} \sum_{j=i+1}^d g_1(E_i,E_j), \\
S_4(G) & = & \sum_{i=1}^{d-1} \sum_{j=i+1}^d g_2(E_i,E_j).
\end{eqnarray*}

First, we compute numbers $n^{0}_i$ and $n^{1}_i$ for any elementary cut. The results are presented in Table \ref{tabela1}.

\begin{table}[h!]
\centering

\begin{tabular}{|c||c|c|} 
\hline
Elementary cut & $n^{0}_{i}$ & $n^{1}_{i}$   \\ \hline \hline
$C_i$ ($i=1,\ldots,n-1$) & $im$ & $(n-i)m$   \\ \hline
$D_i$ ($i=1,\ldots,m-1$) & $ni$ & $n(m-i)$   \\ \hline

\end{tabular}
\caption{ \label{tabela1} Number of vertices in the connected components with respect to an elementary cut.}
\end{table}

\noindent
Therefore, we can compute the following contributions from Table \ref{tabela3}.

\begin{table}[h!]
\centering

\begin{tabular}{|c||c|} 
\hline
 Sum & Result     \\ \hline \hline
$\sum_{i=1}^{n-1} f_1(C_i)$, $n \geq 2$ & $\frac{1}{6}(m^2n^3 - m^2n)$    \\ \hline
$\sum_{i=1}^{m-1} f_1(D_i)$, $m \geq 2$ & $\frac{1}{6}(n^2m^3 - n^2m)$       \\ \hline \hline
$\sum_{i=1}^{n-1} f_2(C_i)$, $n \geq 2$ & $\frac{1}{6}(m^3n^4 - m^3n^2 - 2m^2n^3 + 2m^2n)$     \\ \hline
$\sum_{i=1}^{m-1} f_2(D_i)$, $m \geq 2$ & $\frac{1}{6}(n^3m^4 - n^3m^2 - 2n^2m^3 + 2n^2m)$ \\ \hline

\end{tabular}
\caption{ \label{tabela3} Partial results for computing $S_1(G_{m,n})$ and $S_2(G_{m,n})$.}
\end{table}

\noindent
Finally, for $m,n \geq 2$ one can obtain
\begin{eqnarray*}
S_1(G_{m,n}) & = & \sum_{i=1}^{n-1} f_1(C_i) + \sum_{i=1}^{m-1}f_1(D_i) \\ 
& = & \frac{1}{6} \big(m^3n^2 + m^2n^3 -m^2n - mn^2 \big),
\end{eqnarray*}

\begin{eqnarray*}
S_2(G_{m,n}) & = & \sum_{i=1}^{n-1} f_2(C_i) + \sum_{i=1}^{m-1}f_2(D_i) \\ 
& = & \frac{1}{6}\big(m^4n^3 + m^3n^4 - 3m^3n^2 - 3m^2n^3 + 2m^2n + 2mn^2 \big).
\end{eqnarray*}

Using $S_1(G_{m,n})$, Theorem \ref{wie}, and Equation \eqref{cut_wie}, we can compute the Steiner 3-Wiener index of $G_{m,n}$.

\begin{proposition}
Let $G_{m,n}$ be a grid graph such that $m,n \geq 2$. Then
$$ SW_3(G_{m,n}) = \frac{1}{12} \big(m^4n^3 + m^3n^4 - 3m^3n^2 - 3m^2n^3 + 2m^2n + 2mn^2 \big).$$
\end{proposition}

Next, we compute numbers $n^{00}_{ij}$, $n^{01}_{ij}$, $n^{10}_{ij}$, and $n^{11}_{ij}$ for any pair of elementary cuts. The results are gathered in Table \ref{tabela2}.

\begin{table}[h!]
\centering
\begin{tabular}{|c||c|c|c|c|} 
\hline
Pair of elementary cuts & $n^{00}_{ij}$ & $n^{01}_{ij}$ & $n^{10}_{ij}$ & $n^{11}_{ij}$  \\ \hline \hline
$C_i, C_j$ ($i<j$) & $im$ & $0$ & $(j-i)m$ & $(n-j)m$  \\ \hline
$D_i, D_j$ ($i<j$) & $ni$ & $0$ & $n(j-i)$ & $n(m-j)$  \\ \hline
$C_i,D_j$  & $ij$ & $i(m-j)$ & $(n-i)j$& $(n-i)(m-j)$  \\ \hline

\end{tabular}
\caption{ \label{tabela2} Number of vertices in the connected components with respect to all the pairs of elementary cuts.}
\end{table}

Hence, we can calculate the following contributions from Table \ref{tabela4}.

\begin{table}[h!]
\centering

\begin{tabular}{|c||c|} 
\hline
 Sum & Result     \\ \hline \hline
$\sum_{i=1}^{n-2}\sum_{j=i+1}^{n-1} g_1(C_i,C_j)$, $n \geq 3$ & $\frac{1}{24}(m^2n^4 - 2m^2n^3 - m^2n^2 + 2m^2n)$    \\ \hline
$\sum_{i=1}^{m-2}\sum_{j=i+1}^{m-1} g_1(D_i,D_j)$, $m \geq 3$ & $\frac{1}{24}(m^4n^2 - 2m^3n^2 - m^2n^2 + 2mn^2)$ \\ \hline
$\sum_{i=1}^{n-1}\sum_{j=1}^{m-1} g_1(C_i,D_j)$, $m,n \geq 2$ & $\frac{1}{18}(m^3n^3 - m^3n - mn^3 + mn)$ \\ \hline \hline
$\sum_{i=1}^{n-2}\sum_{j=i+1}^{n-1} g_2(C_i,C_j)$, $n \geq 3$ & $\frac{1}{120}(7m^3n^5 - 10m^3n^4 - 15m^3n^3 - 10m^2n^4$ \\ 
& $+ 10m^3n^2 + 20m^2n^3 + 8m^3n + 10m^2n^2 - 20m^2n)$     \\ \hline
$\sum_{i=1}^{m-2}\sum_{j=i+1}^{m-1} g_2(D_i,D_j)$, $m \geq 3$ & $\frac{1}{120}(7m^5n^3 - 10m^4n^3 - 15m^3n^3 - 10m^4n^2$ \\ 
&$+ 10m^2n^3 + 20m^3n^2 + 8mn^3 + 10m^2n^2 - 20mn^2)$ \\ \hline
$\sum_{i=1}^{n-1}\sum_{j=1}^{m-1} g_2(C_i,D_j)$, $m,n \geq 2$ & $\frac{1}{9}(m^4n^4 - m^4n^2 - m^3n^3 - m^2n^4 + m^3n + m^2n^2 + mn^3 - mn)$\\ \hline

\end{tabular}
\caption{ \label{tabela4} Partial results for computing $S_3(G_{m,n})$ and $S_4(G_{m,n})$.}
\end{table}

\noindent
Finally, for $m,n \geq 3$ one can compute

\begin{eqnarray*}
S_3(G_{m,n}) & = & \sum_{i=1}^{n-2}\sum_{j=i+1}^{n-1} g_1(C_i,C_j) + \sum_{i=1}^{m-2}\sum_{j=i+1}^{m-1} g_1(D_i,D_j) + \sum_{i=1}^{n-1}\sum_{j=1}^{m-1} g_1(C_i,D_j) \\
& = & \frac{1}{24} \big(m^4n^2 + m^2n^4 - 2m^3n^2 - 2m^2n^3 - 2m^2n^2  + 2m^2n + 2mn^2 \big) \\
& + & \frac{1}{18} \big(m^3n^3 - m^3n - mn^3 + mn \big),
\end{eqnarray*}

\begin{eqnarray*}
S_4(G_{m,n}) & = & \sum_{i=1}^{n-2}\sum_{j=i+1}^{n-1} g_2(C_i,C_j) + \sum_{i=1}^{m-2}\sum_{j=i+1}^{m-1} g_2(D_i,D_j) + \sum_{i=1}^{n-1}\sum_{j=1}^{m-1} g_2(C_i,D_j) \\
& = & \frac{1}{120} \big(7m^5n^3 + 7m^3n^5 - 10m^4n^3 - 10m^3n^4 - 10m^4n^2 - 30m^3n^3 - 10m^2n^4 \\
&  + & 30m^3n^2 + 30m^2n^3 + 8m^3n + 20m^2n^2 + 8mn^3 - 20m^2n - 20mn^2 \big) \\
& + & \frac{1}{9} \big(m^4n^4 - m^4n^2 - m^3n^3 - m^2n^4 + m^3n + m^2n^2 + mn^3 - mn \big).
\end{eqnarray*}

Combining all the obtained results and Theorem \ref{glavni}, we obtain the main result of this section.

\begin{theorem}
Let $G_{m,n}$ be a grid graph such that $m,n \geq 3$. Then
\begin{eqnarray*}
SWW_3(G_{m,n})&=& \frac{1}{360} \big( 9m^5n^3 + 15m^4n^4 + 9m^3n^5 + 15m^4n^3 + 15m^3n^4 - 30m^4n^2 - 50m^3n^3 - 30m^2n^4 \\
&+ & 26m^3n - 45m^3n^2 - 45m^2n^3  + 45m^2n^2 + 26mn^3 + 30m^2n + 30mn^2 - 20mn \big).
\end{eqnarray*}
\noindent
Moreover, for any $n\geq 3$ it holds
\begin{eqnarray*}
SWW_3(G_{2,n})&=&\frac{1}{15} \big(3n^5 + 10n^4 - 25n^2 + 12n \big).
\end{eqnarray*}

​
\end{theorem}

\section*{Acknowledgement} 
\noindent
The author was financially supported from the Slovenian Research Agency (research core funding No. P1-0297).

\end{document}